\documentclass[a4paper,10pt,abstracton]{scrartcl}

\usepackage[margin=25mm]{geometry}
\usepackage[utf8]{inputenc}
\usepackage{todonotes}
\usepackage{amsmath}
\usepackage{amsthm}
\usepackage{amssymb}
\usepackage{algorithm,algorithmicx}
\usepackage{algpseudocode}
\usepackage{float}
\usepackage{tikz}
\usetikzlibrary{automata}
\usetikzlibrary{arrows,shapes}

\usepackage[USenglish]{babel}

\usepackage{graphicx}

\usepackage{amssymb,amsmath,verbatim,tabularx}
\usepackage{bbm}
\usepackage{colonequals,graphicx,psfrag}
\usepackage{enumerate} 
\usepackage{subfigure}

\newtheorem{theorem}{Theorem}

\newcommand{\R}{\mathbb{R}}
\newcommand{\Z}{\mathbb{Z}}

\usepackage{authblk}

\begin{document}

\title{Data-driven Prediction of Relevant Scenarios for Robust Combinatorial Optimization}

\author[1]{Marc Goerigk\footnote{marc.goerigk@uni-siegen.de, corresponding author, supported by the Deutsche Forschungsgemeinschaft (DFG) through grant GO 2069/1-1}}
\author[2]{Jannis Kurtz\footnote{j.kurtz@uva.nl}}

\affil[1]{Network and Data Science Management, University of Siegen, Unteres Schlo{\ss} 3, 57072 Siegen, Germany}
\affil[2]{Amsterdam Business School, University of Amsterdam, Plantage Muidergracht 12, 1018 TV, Amsterdam, Netherlands}

\date{}

\maketitle

\begin{abstract}
We study iterative methods for (two-stage) robust combinatorial optimization problems with discrete uncertainty. We propose a machine-learning-based heuristic to determine starting scenarios that provide strong lower bounds. To this end, we design dimension-independent features and train a Random Forest Classifier on small-dimensional instances. Experiments show that our method improves the solution process for larger instances than contained in the training set and also provides a feature importance-score which gives insights into the role of scenario properties.
\end{abstract}


\section{Introduction}
Optimization under uncertainty is an important research field especially due to its relevance in practical applications from operations research. In the real world many parameters of an optimization problem can be uncertain, e.g. the demands, returns or traffic situations or any other parameters which are not precisely known due to measurement or rounding errors. It was shown that hedging against possible perturbations in the problem parameters is essential, since already small perturbations can lead to a large violation of the constraints \cite{ben2009robust}. Driven by the seminal works \cite{soyster1973,kouvelis,roconv,bentalRobust,BertsimasS04} robust optimization evolved to be one of the most popular approaches to tackle uncertainty in optimization problems by finding solutions which are worst-case optimal and feasible for all parameters of a pre-defined uncertainty set; see \cite{bertsimas2011theory,buchheimrobust,gabrel2014recent} for a literature overview. Later the classical robust optimization approach was extended to the two-stage robust optimization approach (also called adaptive robust optimization) in \cite{ben2004adjustable} which has been extensively studied from then on; see e.g. \cite{bertsimas16_2,PostekH16,hanasusanto2015k,kammerling2020oracle,goerigk2019robust,yanikouglu2018survey}.
While most of the works mentioned above consider the case of convex uncertainty sets, discrete uncertainty sets naturally appear in practical applications where often a finite set of historical observations of the uncertain parameters is given; see \cite{aissi_minmax_survey,buchheimrobust,kouvelis}. Based on these historical observations uncertainty sets can be constructed by data-driven approaches from statistics or machine learning \cite{goerigk2020data,chassein2019algorithms,shang2017data,bertsimas2018data,cheramin2021data}.

While dualizing the robust objective function often leads to compact reformulations of the robust optimization problem, in some cases this derivation is not possible (e.g. in the two-stage case with integer recourse variables or for discrete uncertainty sets) or not desirable (see \cite{fischetti2012cutting}). In this case the robust optimization problem can alternatively be solved by scenario generation approaches \cite{borrero2021modeling} while the two-stage robust problem can be solved by column-and-constraint algorithms \cite{zeng2013solving}. In both cases the idea is to iteratively generate new scenarios by solving the adversarial problem and adding them to a master problem. While both approaches often perform well, one of the main drawbacks is that with an increasing number of iterations the size of the master problem increases, which often leads to problem sizes which cannot be solved in appropriate time; see \cite{kammerling2020oracle}. Hence an important tool to improve the efficiency of the iterative methods is the choice of a good starting set of scenarios. 

In this work we introduce what we call the \textit{Relevant Scenario Recognition Problem} which aims at finding a set of $k$ scenarios which maximizes the lower bound given by the master problem of the considered iterative algorithm. This is motivated by the fact that starting the iterative algorithms with the largest possible lower bound leads to a small optimality gap and to a reduced number of iterations. Additionally we can ensure that even if the algorithm does not terminate in an appropriate number of iterations we find a solution with small optimality gap after the first iteration.     

\textit{Our contributions are the following.} We study (two-stage) robust combinatorial optimization problems with objective uncertainty and introduce the Relevant Scenario Recognition Problem, where we aim at finding the strongest subset of scenarios in the master problem of a given size. We show that this problem is NP-hard and that it can be reformulated as a linear binary program. Since solving the problem exactly is intractable we show how to create a set of dimension-independent features which can be used to train Random Forests on a set of already solved small-dimensional problem instances. In computational experiments using two-stage shortest path and traveling salesperson problems with uncertain costs, we evaluate the effect of our method in comparison with random starting scenarios and starting scenarios that maximize the sum of costs and show that our method improves the computational performance even for larger-dimensional instances of the problem than contained in the training set. Our experiments show that the selection of the right starting scenarios is crucial for the subsequent behavior of the iterative method. Additionally analyzing the feature importance our method gives new insights into the important properties of scenarios in robust optimization.

\section{Preliminaries}\label{sec:preliminaries}

\subsection{Robust Optimization}
We consider deterministic optimization problems of the form
\begin{equation}\label{eq:deterministicProblem}\tag{P}
\begin{aligned}
\min \ & c^\top x \\ 
s.t. \quad & x \in X 
\end{aligned}
\end{equation}
where $c\in\R^n$ is a given cost vector and $X \subset \R^n$ an arbitrary feasible region. 
Note that we can choose $X\subset \Z^n$ to model integer problems and for combinatorial optimization problems we often have $X\subset \{ 0,1\}^n$. We assume that the parameters of the cost vector $c$ are uncertain and all possible realizations are given by a discrete uncertainty set $U=\left\{ c^1, \ldots ,c^m\right\}$ where each $c^i\in \R^{n}$ is called a \textit{scenario}. The aim in the classical robust optimization approach is to calculate a solution $x^*\in X$ which minimizes the worst-case objective value over all scenarios in $U$, i.e. we want to find an optimal solution of problem
\begin{equation}\label{eq:robustProblem}\tag{RP}
\begin{aligned}
\min_{x\in X} \ & \max_{c\in U} \ c^\top x .
\end{aligned}
\end{equation}

Note that uncertainty in the objective is a common modeling assumption for combinatorial problems. Corresponding robust combinatorial optimization problems with discrete uncertainty are known to be computationally hard, even if 
we have only two scenarios; see \cite{kouvelis,buchheimrobust}. Despite this theoretical drawback these problems are often solved successfully using an iterative scenario generation approach. In this algorithm we alternately calculate an optimal solution $x^*\in X$ of problem~\eqref{eq:robustProblem} for a subset of scenarios $U'\subset U$, which we call the \textit{master problem}:
\begin{equation}\label{eq:masterproblemRO}\tag{MP}
\min_{x\in X} \ \max_{c\in U'} \ c^\top x.
\end{equation}
Afterwards the \textit{adversarial problem} is solved to find a new scenario $c^*\in U$ which maximizes the objective value $c^\top x^*$, i.e. we solve
\begin{equation}\label{eq:adversarialproblemRO}\tag{AP}
\max_{c\in U} \ c^\top x^*
\end{equation}
and update $U':= U' \cup \{ c^*\}$. We iterate in this way until we cannot find any scenario anymore in $U$ which improves the objective value of \eqref{eq:masterproblemRO} in which case the current solution $x^*$ is robust optimal. Finally note that the optimal values of \eqref{eq:masterproblemRO} define a non-decreasing sequence of lower bounds for the optimal value of \eqref{eq:robustProblem} while the optimal values of \eqref{eq:adversarialproblemRO} are upper bounds. Hence at each iteration an optimality gap is provided.

\subsection{Two-stage Robust Optimization}
In the two-stage setting we consider deterministic optimization problems of the form
\begin{equation}\label{eq:deterministicProblem2}\tag{2P}
\begin{aligned}
\min \ & c^\top x + d^\top y \\ 
s.t. \quad & Ax + D y  \le b \\
& x\in X, y\in Y 
\end{aligned}
\end{equation}
where $c\in\R^n, d\in\R^q$ are given cost vectors, $X\subset\R^{n}$ and $Y\subset \R^q$ are arbitrary sets which can be integer or binary, $A\in\R^{p\times n}$, $D\in \R^{p\times q}$ and  $b \in \R^p$. We assume that only the parameters of cost vector $d$ are uncertain. The variables $x\in \R^n$ are the first-stage solutions which have to be fixed \textit{here-and-now}. Variables $y\in \R^q$ are the second-stage variables, also called \textit{wait-and-see} variables, which can be adjusted flexibly after all uncertain parameters are known. Again we assume that all possible realizations of the uncertain cost parameters are given by a discrete uncertainty set $U=\left\{ d^1, \ldots ,d^m\right\}$ where $d^i\in\R^{q}$ are again called \textit{scenarios}. The aim in two-stage robust optimization is then to calculate a first-stage solution $x^*$ such that for every scenario $d^i\in U$ there exists a feasible second-stage solution $y^i$ such that the worst possible objective value over all scenarios is minimized. This can be modeled by problem 
\begin{equation}\label{eq:robust2StageProblem}\tag{2RP}
\begin{aligned}
\min \ & c^\top x + \mu \\ 
s.t. \quad & (d^i)^\top y^i \le \mu \quad \forall \ i\in [m]\\
& A x + D y^i\le b \quad \forall \ i\in [m]\\
& \mu\in\R, \ x\in X, \ y^i\in Y \ \forall \ i\in [m] .
\end{aligned}
\end{equation}
where we use the notation $[m]:=\left\{ 1,\ldots ,m\right\}$ for each $m\in \mathbb{N}$.

One of the most popular methods to solve two-stage robust problems is the column-and-constraint generation algorithm (CCG), first introduced in \cite{zeng2013solving} which can also be used in the case of discrete uncertainty sets. Similar to the scenario generation procedure described in the previous subsection, the idea is to iteratively generate new scenarios and add them to the master problem. Here each time we add a scenario we also have to add a new set of variables $y^i$. More precisely, we alternately calculate an optimal solution $(x^*,\mu^*)$ of problem \eqref{eq:robust2StageProblem} for the current subset of scenarios given by indices $\mathcal I\subset [m]$, which we call the \textit{master problem}, and afterwards an \textit{adversarial problem} is solved to find a new scenario $d^*$ which maximizes the objective value, i.e. we solve
\begin{equation}\label{eq:adversarialProblem2Stage}
\max_{d\in U} \ \min \left\{ d^\top y : Ax^* + Dy \le b, \ y\in Y\right\}  .
\end{equation}
Afterwards we add the index of $d^*$ to $\mathcal I$ and we iterate. If the optimal value of \eqref{eq:adversarialProblem2Stage} is smaller or equal to $\mu^*$, then the current solution $x^*$ is optimal for all scenarios in $U$. Finally, note that as in the one-stage case the optimal values of the master problems define a non-decreasing sequence of lower bounds for the original robust problem \eqref{eq:robustProblem}. On the other hand, the optimal value of the adversarial problem always returns the objective value of solution $x^*$ and is therefore a valid upper bound. Hence in each iteration of the CCG we can calculate an optimality gap to check how far we are at most from being optimal.

\section{The Relevant Scenario Recognition Problem}  
Solving the robust problem \eqref{eq:robustProblem} and especially the two-stage robust problem \eqref{eq:robust2StageProblem} can be computationally challenging. In several methods (e.g. the constraint generation or column-and-constraint generation method described in Section \ref{sec:preliminaries}) we iteratively solve a master problem containing a subset of scenarios $\mathcal I\subset [m]$ and afterwards an adversarial problem. While often the adversarial problem can be solved more efficiently, the master problem can be very challenging after some iterations since in each iteration we have to increase the size of the formulation by adding a new scenario. The increasing computational effort can be observed especially in the two-stage case; note that here we also have to add a new set of second-stage variables $y^i$. 
Even few iterations may lead to master problems which cannot be solved in appropriate time (see e.g. \cite{goerigk2020min,kammerling2020oracle}). 

Hence an important ingredient influencing the performance of the iterative methods is the selection of a ``good'' set of starting scenarios. If we can find a small set of starting scenarios which lead to a small optimality gap in the first iteration, then even if the algorithm gets stuck after a few iterations, we already have a solution with small optimality gap. To achieve this the idea is to find scenarios which maximize the lower bound given by the master problem in each iteration of the iterative method. More precisely, for a given $k\le m$ we want to find a set of $k$ starting scenarios such that the optimal value of the corresponding master problem is maximized, which we call the \textit{Relevant Scenario Recognition Problem} (RSRP). In the following, we study this problem for the robust optimization problem and for the two-stage robust optimization problem. 

We assume we have a given discrete uncertainty set $U=\left\{ c^1,\ldots ,c^m\right\}$. In the scenario generation method described in Section \ref{sec:preliminaries} we always obtain a lower bound to the original robust problem \eqref{eq:robustProblem} by solving the master problem \eqref{eq:masterproblemRO} 
for the current set of scenario indices $\mathcal I\subset [m]$. For a given $k\le m$ the RSRP is then given by
\begin{equation}\label{eq:BSAP_RO}\tag{RSRP-RO}
\max_{\substack{\mathcal I \subset [m]\\ |\mathcal I |\le k}} \ \min_{x\in X} \ \max_{i\in\mathcal I}  \ (c^i)^\top x .
\end{equation}

We first study the complexity of problem~\eqref{eq:BSAP_RO}.
\begin{theorem}\label{thm:complexityBSAP}
Let $X$ be a polyhedron given by an outer description. Problem~\eqref{eq:BSAP_RO} can be solved in polynomial time if $k$ is a constant value, but is NP-hard if $k$ is part of the input.
\end{theorem}
\begin{proof}
Note that there are at most $m^k$ possible subsets $\mathcal{I}$ of $[m]$. Hence, if $k$ is a constant value, problem~\eqref{eq:BSAP_RO} can be solved by solving a polynomial number of continuous robust problems with a polynomial number of scenarios. The latter type of problems can be solved in polynomial time.

For the case that $k$ is part of the input, consider any instance of the Set Cover problem, which is known to be NP-hard \cite{gareyjohnson}. Given a set of items $E=\{e_1,\ldots,e_N\}$ and sets $S_i\subset E$ for $i=1,\ldots,M$, we need to decide if there exists a set of indices $\mathcal{I}\subset[M]$ with cardinality at most $K$ such that each item from $E$ is contained in at least one set $S_i$, $i\in\mathcal{I}$.

Set $n=N$, $m=M$, $k=K$, $X=\left\{x\ge 0 \ | \ \mathbbm{1}^\top x=1\right\}$ where $\mathbbm{1}$ is the all-one vector. We define the scenario $c^i\in \R^n$ such that the $j$th entry is equal to one if and only if $e_j \in S_i$ and it is equal to zero otherwise. Note that if a selection of sets $\mathcal I$ does not cover $E$ than an optimal solution of the inner minimization problem in \eqref{eq:BSAP_RO} can be derived by setting $x_i=1$ for a non-covered item $i$ resulting in an optimal value of zero. On the other hand, if $E$ is covered the optimal value must be strictly larger than zero. Hence it holds that there is a solution to problem~\eqref{eq:BSAP_RO} with objective value strictly larger than zero if and only if there exists a set cover with cardinality at most $K$.
\end{proof}

Note that, if $X$ is a polyhedron, we may dualize the inner minimization problem after performing a level set transformation to shift the inner maximization to the constraints. Modeling the selection of the index set $\mathcal I$ by binary variables we can then reformulate the problem as a compact linear problem formulation by using standard linearization techniques for the product of binary and dual variables. On the other hand, if $X$ is a finite set the following theorem holds, which is proved in the appendix.
\begin{theorem}\label{thm:exactIPformulationRO}
If $X$ is finite, problem \eqref{eq:BSAP_RO} can be reformulated as a binary linear program.
\end{theorem}

For the two-stage robust case we again assume we have a given discrete uncertainty set $U=\left\{ d^1,\ldots ,d^m\right\}$. In the column-and-constraint generation method described in Section \ref{sec:preliminaries} we always obtain a lower bound to the original robust problem \eqref{eq:robustProblem} by solving the master problem which is given by \eqref{eq:robust2StageProblem} for the current set of scenarios $\mathcal I\subset [m]$. For a given $k\le m$ the RSRP is then given by
\begin{equation}\label{eq:BSAP_2StageRO}\tag{RSRP-2RO}
\max_{\substack{\mathcal I \subset [m]\\ |\mathcal I |\le k}} \ \min_{\substack{A x + D y^i\le b \ \forall i\in \mathcal I \\ x\in X, y^i\in Y \ \forall i\in \mathcal I}}\ \max_{i\in \mathcal I} \ \left( c^\top x +  (d^i)^\top y^i\right) .
\end{equation}

It follows directly from Theorem \ref{thm:complexityBSAP} that problem \eqref{eq:BSAP_2StageRO} is NP-hard under the same assumptions since we can easily reduce problem \eqref{eq:BSAP_RO} to \eqref{eq:BSAP_2StageRO}.

Note that similar to the single-stage case, \eqref{eq:BSAP_2StageRO} can be reformulated as a binary linear program, which is proved in the appendix. 
\begin{theorem}\label{thm:exactIPformulation2StageRO}
If $X$ is finite, problem \eqref{eq:BSAP_2StageRO} can be reformulated as a binary linear program.
\end{theorem}

\section{Data-driven Heuristic}
\label{sec:pearson}

\subsection{Prediction of Relevant Scenarios}

In this section we present a machine-learning-based method to extract information from given training data to find good solutions to problem \eqref{eq:BSAP_RO} (or \eqref{eq:BSAP_2StageRO}). We assume we have a given set of $N$ training samples where each training sample consists of an instance of problem \eqref{eq:robustProblem} (or \eqref{eq:robust2StageProblem}). In case of classical robustness \eqref{eq:robustProblem} each training instance is given by a tuple $(U,X)$ where $U=\left\{ c^1,\ldots ,c^m\right\}$ is a finite uncertainty set and $X$ the feasible set. We denote the set of training samples as $\mathcal T_{RO}=\left\{(U^1,X^1),\ldots , (U^N,X^N) \right\}$. In the two-stage robust case each instance is given by a $7$-tuple $(c,U,X,Y,A,D,b)$ where $c\in\R^n$ is the first-stage cost vector, $U=\{ d^1,\ldots ,d^m\}$ is the uncertainty set, $X,Y$ are the feasible sets of first- and second-stage respectively, $A,D$ are the constraint matrices and $b$ the right-hand-side vector in problem \eqref{eq:deterministicProblem2}. We denote the set of training samples as $\mathcal T_{2RO}=\left\{(c^1,U^1,X^1,Y^1,A^1,D^1,b^1),\ldots , (c^N,U^N,X^N,Y^N,A^N,D^N,b^N) \right\}$. We assume that each training instance is labeled, i.e. for each instance $j\in [N]$ we have a target vector $\lambda^j\in \{ 0,1\}^{U^j}$ where $\lambda_i^j=1$ if scenario $d^i\in U^j$ is a \textit{relevant scenario} and $\lambda_i^j=0$ otherwise. Which scenarios are defined as relevant and how these scenarios are determined is a design choice which has to be made by the user. We discuss several options in Section~\ref{sec:Guidelines}. We denote the set of corresponding label vectors for the training samples in $\mathcal T_{2RO}$ by $\Lambda=\{ \lambda^1,\ldots ,\lambda^N\}$. In the following all results are presented for the two-stage robust problem. Nevertheless all results can easily be adapted for classical robust optimization.

Assume we have a given instance of problem \eqref{eq:robust2StageProblem} with cost vectors $c$ and uncertainty set $U$ which we want to solve. The idea of our approach is to train a Random Forest Classifier (RFC) on the labeled training data in $\mathcal T_{2RO}$ to detect relevant scenarios in an uncertainty set. We then use the trained RFC to predict an output for uncertainty set $U$ which is a target vector $\lambda\in [0,1]^U$ where a large entry of $\lambda$ denotes that the corresponding scenario is more likely a relevant scenario. The latter output vector is then normalized and interpreted as a probability distribution over the scenarios in $U$. We pick $k$ scenarios in $U$ randomly regarding the derived probability distribution which gives us a feasible solution for problem \eqref{eq:BSAP_2StageRO}.

The main question is how to represent an instance from our training set $\mathcal T_{2RO}$ as a feature vector such that an RFC can derive useful predictions. Two important properties our representation should have is that it is \textit{independent of the dimensions} $n,q$ of problem \eqref{eq:deterministicProblem2} and the \textit{number of scenarios} contained in the uncertainty sets. Being independent of the dimension of the problem is important since we want to train RFC on small-sized instances and obtain predictions for larger instances. Being independent of the number of scenarios contained in the uncertainty set is important since the size of the uncertainty set can vary between instances. We achieve both properties by designing features for each scenario which are independent of the dimension of the problem. Using expert knowledge about robust optimization we are able to incorporate important information about the structure of the scenario itself, the relation to the uncertainty and the structure of the underlying problem into the feature design.

\subsection{Feature Design}
In this subsection we present the list of features we designed to obtain a dimension-independent representation of each scenario in our training instances. More precisely, for each instance $j\in [N]$ and each scenario $d^i\in U^j$ we generate the following list of features. We group the features into three categories. The first category contains features which only consider inherent properties of the single scenario $d^i$. The second category contains features which relate the scenario $d^i$ to all other scenarios of the uncertainty set $U^j$. Finally the third category relates the scenario $d^i$ to the underlying problem structure given by the feasible sets $X$,$Y$ and constraint-parameters $A,D,b$.

In the following consider a given scenario $d^i\in\R^q$ contained in the uncertainty set $U^j$ of instance $j\in [N]$ with first-stage cost vector $c$, feasible regions $X,Y$ and constraint parameters $A,D,b$.
\paragraph{Category 1: Single Scenario Features}
\begin{enumerate}
\item Calculate the average value of the absolute values of the entries of $d^i$
\[
f_{1,1}(d^i):=\frac{1}{q}\sum_{l=1}^{q} |d_l^{i}| .
\]
This is motivated by the observation that scenarios with large absolute entries are likely candidates for worst-case scenarios. 

\item Calculate the variance of the entries of $d^i$
\[
f_{1,2}(d^i):=\frac{1}{q}\sum_{l=1}^{q} (d_l^{i}-\bar d^{i})^2
\]
where $\bar d^{i}$ is the average value over all entries of scenario $d^{i}$. A high variance of a scenario can lead to significant variations in objective value of the scenario over different solutions. Hence this is a relevant information about the performance as a worst-case scenario.

\item Calculate the maximum entry of the absolute values of the scenario
\[
f_{1,3}(d^i):=\max_{l=1,\ldots ,q} |d_l^{i}| .
\]
Scenarios with large maximum entries lie close to the boundary of the box containing all scenarios of an uncertainty set. Hence they are more likely to be considered as worst-case scenarios.
\end{enumerate}

\paragraph{Category 2: Scenario Features Related to the Uncertainty Set}
\begin{enumerate}
\item Calculate the distance to the center point of the uncertainty set
\[
f_{2,1}(d^i):=\|d^{i}-\bar d\| 
\]
where \[\bar d = \frac{1}{|U^j|} \sum_{d\in U^j} d\] is the center of the uncertainty set and $\|\cdot\|$ the Euclidean norm. The distance to the center point contains relevant information since it contains information about the geometrical location of the scenario in the uncertainty set, namely if it is located deep in convex hull of the uncertainty set or more on the boundary.
\item Calculate the average quadratic distance to center point of the uncertainty set
\[
f_{2,2}(d^i):=\|d^{i}-\bar d\|^2 .
\]
The motivation is the same as for feature $f_{2,1}$, but here large distances get reinforced.

\item Calculate the average distance to the closest $\kappa$ scenarios
\[
f_{2,3}(d^i):=\frac{1}{\kappa}\sum_{t\in D_i^\kappa} \|d^{i}-d^{t}\|
\]
where $D_{i}^\kappa$ contains the indices of the $\kappa$ closest scenarios from $U^j\setminus\{d^i\}$ to scenario $d^{i}$. This feature contains information about the density of the region in which the scenario is located. If it lies in a dense region of the uncertainty set the average distance to the closest scenarios is small while it is large if $d^i$ lies in a region with low density. Note that this features also provides a score for considering the scenario as an outlier. Parameter $\kappa$ is an input to this feature.

\item Calculate the scalar product with the center of the uncertainty set 
\[
f_{2,4}(d^i):=\bar d^\top d^{i}
\]
where $\bar d$ is the center of the uncertainty set $U^j$. The scalar product does not only contain information about the norm of $d^i$ and $\bar d$ but also about the angle between both vectors and hence about the location of $d^i$ related to the uncertainty set.
\end{enumerate}

\paragraph{Category 3: Scenario Features Related to the Problem Structure}
\begin{enumerate}
\item Calculate the optimal value of the deterministic problem \eqref{eq:deterministicProblem2} for scenario $d^i$ (denoted by opt$(d^i)$)
\[
f_{3,1}(d^i):= \text{opt}(d^i).
\]
A larger deterministic optimal value points to a scenario which is more likely relevant since for all first-stage and second-stage solutions this scenario provides a large objective value. Note that this feature relates the scenario to the structure of the problem and the first-stage solutions. If for example a certain variable is never used by all feasible solutions then the corresponding scenario-entry does not influence this feature, in contrast to the features in Category $1$ and $2$. 

\item Calculate the objective value of a deterministic solution for scenario $d^i$ for the average scenario
\[
f_{3,2}(d^i):=\bar d^\top y^i
\]
where $\bar d$ is the average scenario of $U^j$ and $y^i$ is an optimal second-stage solution of \eqref{eq:deterministicProblem2} for scenario $d^i$. Since $\bar d$ contains information about all scenarios in $U^j$, this feature connects the deterministic solutions to all other scenarios, i.e. the structure of the problem is connected to the uncertainty set.

\item Define a distance measure between two scenarios, which takes into account the underlying problem structure. To this end let $(x^i,y^i)$ be an optimal solution of the deterministic problem with first-stage costs $c$ and second-stage costs $d^i$. We define a distance measure between scenario $d^i$ and $d^j$ by
\[
\delta(d^i,d^j):= \frac{1}{2}\left( (d^i)^\top y^i - (d^i)^\top y^j \right) + \frac{1}{2} \left( (d^j)^\top y^j - (d^j)^\top y^i \right).
\]
The distance is small if the optimal solution of scenario $d^i$ has an objective value close to the optimal one in scenario $d^j$ and vice versa. Hence a small distance indicates that both solutions perform similarly in their corresponding scenarios. In contrast to the Euclidean distance we use for the features in Category $1$ and $2$ this distance includes the underlying problem structure of the problem. We can now generate features $f_{2,1},f_{2,2},f_{2,3}$ where we replace the Euclidean distance by distance $\delta$. We denote these features by $f_{3,3},f_{3,4},f_{3,5}$.

\item Calculate the optimal first-stage solution if all second-stage cost-scenarios would be zero, and measure its performance in each scenario of $U^j$. More precisely, calculate an optimal first-stage solution $x^*$ of the problem 
\[
\begin{aligned}
\min \ & c^\top x \\ 
s.t. \quad & Ax + D y  \le b \\
& x\in X, y\in Y 
\end{aligned}
\]
and define the feature as
\[
f_{3,6}(d^i):= (d^i)^\top x^* .
\]
This feature connects information about the first-stage solution to the scenarios. Note that this approach is only possible if $n=q$; we give examples for this assumption in the experimental section.

\item Identify the $\alpha$ most important second-stage variables by the following procedure. Calculate the average scenario $\bar d$ (as above) and solve the deterministic problem in the average scenario after removing second-stage variable $y_t$, i.e. solve
\begin{equation*}
\begin{aligned}
\min \ & c^\top x + \bar d^\top y \\ 
s.t. \quad & Ax + D y  \le b \\
& y_t=0 \\
& x\in X, y\in Y 
\end{aligned}
\end{equation*}
and denote its optimal value as opt$(t)$. If the resulting problem is infeasible we assign a big-M value as optimal value. Note that if the deterministic problem is defined on a graph, the latter procedure would result in removing the corresponding edge from the graph.  Afterwards we calculate the $\alpha$ variable indices $t_1,\ldots , t_\alpha$ with the largest optimal values when removing the corresponding variable. These variables are considered as the most important ones since removing them leads to the largest increase in objective value. We define a feature for scenario $d^i$ as the absolute value of the scenario entry of the corresponding variable, i.e.
\[
f_{3,7}(d^i):=|d_{t_\alpha}^i|
\]
All these features consider the relation between scenario $d^i$ and the underlying problem structure and the first-stage solutions. Note that if a variable is important and the corresponding entry of the scenario has a large absolute value, then this scenario is likely a more relevant scenario for the worst-case. Parameter $\alpha$ is an input to this feature.

\end{enumerate}
Note that the features in Category $3$ lead to heavier computations since they involve solving the deterministic problem \eqref{eq:deterministicProblem2} a large amount of times.

We calculate all features from Categories 1-3 for all scenarios $d^i$ in all uncertainty sets $U^1,\ldots ,U^N$ in our training set. The corresponding features are concatenated forming one feature vector for each scenario. All features from Categories 1-3 are normalized to values in $[0,1]$ by applying min-max normalization, i.e. for each feature $f$ we calculate the minimum value $f_{min}$ and the maximum value $f_{max}$ and define the normalized feature for scenario $d^i$ as 
\[
\bar f(d^i)=\frac{f(d^i)-f_{min}}{f_{max}-f_{min}}.
\]
Then every feature vector gets assigned the corresponding label in $\Lambda$. This generated data can be used to train an RFC.

\subsection{Heuristic Method}

We combine all pre-described concepts into one heuristic algorithm which is shown in Algorithm \ref{alg:heuristic}. The main idea is to go through all scenarios in our training set $\mathcal T_{2RO}$ and calculate all features described in the previous section. Then the corresponding feature vectors are used to train an RFC. If we want to solve a new instance of \eqref{eq:deterministicProblem2} we generate the same feature vectors for all scenarios in the given uncertainty set $U$ and let the trained RFC predict a value in $[0,1]$ for each scenario. These values are normalized to obtain a probability distribution and afterwards we pick $k$ scenarios from $U$ randomly following the given probability distribution. In the normalization, we used squared weights to give increased probability for scenarios where the prediction is close to 1.

\begin{algorithm}\caption{(Data-driven Heuristic)}\label{alg:heuristic}
\begin{algorithmic}
\State {\bfseries Input:} $k,\alpha\in \mathbb N$, instance to solve $(c,U,X,Y,A,D,b)$, training set $\mathcal T_{2RO}$ with corresponding labels $\Lambda$ 
	\State {\bfseries Output:} $k$ scenarios in $U$
	\State set $\mathcal V=\{ \}$ \Comment{Training}
	\For{$j=1,\ldots ,N$ }
	\For{$d\in U^j$}
	\State calculate feature vector
	\[
	f(d):=\left(f_{11}(d),\ldots ,f_{3,6+\alpha}(d)\right)
	\]
	\State $\mathcal V = \mathcal V \cup \{ (f(d), \Lambda(d)) \}$
	\EndFor
	\EndFor
	\State train RFC on $\mathcal V$
	\For{$d\in U$} \Comment{Prediction}
	\State calculate feature vector $f(d)$ as above
	\State calculate prediction $w_d\in [0,1]$ of RFC for $f(d)$
	\EndFor
	\State Define $\tilde w_d = \frac{w_d^2}{\sum_{d\in U} w_d^2}$ for every $d\in U$.
	\State Sample randomly $k$ scenarios without replacement from $U$ following probability distribution $\tilde w$.
	\State \textbf{Return:} $k$ drawn scenarios
\end{algorithmic}
\end{algorithm}

\subsection{Guidelines for Application}\label{sec:Guidelines}
The performance of the data-driven heuristic presented in Algorithm \ref{alg:heuristic} heavily depends on the quality of the training data which is accessible. While for practical applications it would be desirable to build up a data base with already solved and labeled instances of the given problem which can be used by our heuristic to solve future instances of the same problem more efficiently, often such a database does not exist. Hence in the following we discuss how to generate and label instances and for which type of applications our procedure is useful.
 
\paragraph{Generation of Training Instances}
If no data basis of problem instances is available, generating random instances is necessary. To this end for each training instance the cost vectors and the set of scenarios have to be generated. Ideally the generated instances should be close to the real instances of the considered problem and hence problem specific generation methods should be derived by the user. For example, when considering disaster management problems where an earthquake or airplane crash can appear in a certain location, one could simulate several incidents in random locations to generate cost scenarios; see e.g. \cite{goerigk2015two}. In the case of uncertain demands, which often appear in inventory, facility location or vehicle routing problems, possible demand scenarios can be drawn from a user defined distribution with support contained in the range which is observed in historic scenarios. Also more sophisticated statistically sample methods could be applied using empirical knowledge about the underlying distribution. If possible expert knowledge about the application should be incorporated into the generation process.

\paragraph{Labeling Scenarios as Relevant} 
Deciding which of the scenarios in a training instance are labeled as relevant is a crucial step when generating training data since data-driven approaches try to detect structural properties which indicate relevance. In this work the aim is to find scenarios which strongly improve the lower bound given by the master problem used in the CCG. Hence the optimal set of relevant scenarios could be determined by calculating the smallest $k\le m$ such that the objective value of problem \eqref{eq:BSAP_2StageRO} does not change anymore and afterwards choose the scenarios selected by the optimal solution $\mathcal I^*$ as relevant scenarios. Note that for each training point, i.e. each instance, we have to run the latter procedure which cannot be done efficiently. 

A second variant is to solve each training instance by the CCG and collect all scenarios which are generated by the adversarial problem during this method. 
It is not necessary to run the CCG until convergence is reached and \eqref{eq:BSAP_2StageRO} is solved to proven optimality; instead, the process can be stopped after a fixed limit on the computation time and the scenarios that have been generated so far are used.
After the algorithm has terminated we check for each collected scenario if the objective value decreases after removing it, in which case we label it as relevant. All other scenarios are labeled as irrelevant. Note that the latter step is required since the set of scenarios generated by the CCG can contain redundant scenarios.

Alternatively we can define a scenario to be relevant if removing it from the uncertainty set leads to a decrease of the optimal value of problem \eqref{eq:robust2StageProblem}.

Note that all the latter approaches are computationally very demanding since they involve solving a set of problems which we aim to solve more efficiently by Algorithm \ref{alg:heuristic}. However our approach can be trained on small-dimensional instances which can still be solved without being supported by our method. Then it can be used to predict relevant scenarios for high-dimensional instances. Our experiments show that this improves the performance of the solution method. 

Since the generation of the data and feature vectors can be computationally costly such data-driven methods are tailored for applications where we have a long preparation phase, which can be used to build up training data, while in the execution phase fast solutions are desired. The latter properties often appear in disaster management problems \cite{goerigk2015two}. 

Note that in all labeling approaches above infeasible or unbounded instances are detected during the solution process and can be removed.

\section{Experiments}\label{sec:experiments}

\subsection{Setup}

In this section we perform several experiments to evaluate the efficiency and effectiveness of the proposed data-driven scenario prediction method introduced in Section \ref{sec:pearson}. To this end, we consider two types of robust two-stage problems with objective uncertainty. Both problems have in common that, given a graph $G=(V,E)$, in the first-stage problem we need to decide which edges of the graph we want to buy. Once this decision is made, we obtain the costs of each edge and afterwards in the second-stage we have to find a feasible solution with minimal costs for a given network problem in the subgraph that was bought in the first stage. We assume we have $m$ scenarios in the discrete uncertainty set $U = \{d^1,\ldots,d^m \}$ that determine the second-stage costs. Our problem~\eqref{eq:robust2StageProblem} can then be modeled as follows:
\begin{align}
\min\ & \sum_{e\in E} c_e x_e + z \\
\text{s.t. } & z \ge  \sum_{e\in E} d^i_e y^i_e & \forall i\in[m] \\
& y^i_e \le x_e & \forall e \in E, i\in[m] \label{con1}\\
& x_e \in\{0,1\} & \forall e \in E \\
& y^i \in Y & \forall i\in[m]
\end{align}
with $Y$ denoting the set of feasible solutions for a given network problem. In the following experiments, we consider traveling salesperson problems (TSP) and shortest path problems (SP). Variable $x_e$ decides whether to buy edge $e$ or not, i.e., for known costs $c_e$, we can make an edge available for the second-stage problem. Due to constraints~\eqref{con1}, only edges bought this way may be used in the second-stage. 

For TSP, we use the classical Miller-Tucker-Zemlin formulation, see \cite{miller1960integer}.
We note that more effective formulations have been developed to solve TSP, but as all scenario generation methods use the same formulation, this is not the focus of our experiment. We construct complete graphs with numbers of nodes in $\{6,7,8,9\}$. 

To model shortest path problems, we use standard flow constraints on each node.
We generate layered graphs, where each node of one layer is connected to all nodes of the next layer. Additionally, the source node $s$ is connected to each node of the first layer, while all nodes of the last layer are connected to the sink node $t$. Each layer has a width of five nodes, while we vary the number of layers in $\{5,6,7,8\}$.

For the two smaller problem sizes (6 and 7 nodes for TSP, 5 and 6 layers for SP), we generate 120 problem instances. Of these, 100 are used for training, and 20 are used for testing. For the two larger problem sizes (8 and 9 nodes for TSP, 7 and 8 layers for SP), we only generate 20 instances for testing. Each instance has $m=500$ scenarios.
For TSP the first-stage costs $c$ are sampled uniformly iid from $[4,6]^E$, while for the second-stage costs, we sample an instance-specific number $K$ uniformly from $\{3,\ldots,8\}$ to create a multi-modal distribution. For each $k\in [K]$, a midpoint vector $\mu^k \in [25,75]^E$ and a maximum deviation vector $\delta^k \in [0.1,0.5]^E$ are sampled uniformly. For each scenario $i$ that we sample, we choose one $k \in [K]$ randomly and choose the cost of each edge $e$ uniformly in $[(1-\delta^k_e)\mu^k_e, (1+\delta^k_e)\mu^k_e]$. For SP the costs are generated in the same way, except first-stage costs $c$, which are sampled uniformly iid from $[5,15]^E$. The choice for the intervals of the first-stage costs are motivated by our observations that making the intervals too large can lead to trivial solutions where the edges bought in the first stage just form one feasible second-stage solution. Additionally, the size of the boundary of the intervals had to be adjusted since for too small first-stage costs it can be beneficial to buy all edges of the graph, while for too large values again only one feasible second-stage solution is bought. 

We apply the data-driven heuristic described as Algorithm~\ref{alg:heuristic} in Section~\ref{sec:pearson}, in the following denoted as \textit{DDH}. We apply a Random Forest Classifier from the sklearn Python library using 100 trees with a maximum depth of 5. The resulting predictions are then used to create a random distribution over the scenarios. To label the training data, we run the CCG method with a time limit of 60 seconds per instance and afterwards test if redundant scenarios can be removed. Note that training our method can be done in advance, well before an optimization problem needs to be solved. In the following, labeling and training times are not included in the time limits.

As benchmark methods we use the following simple heuristics to create alternative probability distributions: In the method denoted by \textit{Random}, each scenario has equal probability. In the method denoted by \textit{Maxsum}, we calculate the weight of a scenario by its sum of cost values $\sum_{e\in E} d^i_e$ and then normalize in the same way as in Algorithm~\ref{alg:heuristic}, i.e., scenarios with a higher sum are more likely to be relevant than scenarios with a smaller sum. Note that our analysis of feature importance showed that the latter benchmark (equivalent to feature $f_{1,1}$) is below the features with the largest impact on the prediction decision, hence it is an easy but very competitive benchmark.

For each type of problem, we conduct two experiments. In the first experiment, the purpose is to determine the strength of the problem formulation when we only use the top $k$ scenarios with highest priority, where we vary $k$ in $\{1,3,\ldots,15\}$. Recall that the objective value of the resulting two-stage robust optimization problem is a lower bound on the actual objective value when using all scenarios. We run each algorithm to determine the $k$ most relevant scenarios, and solve the resulting problem formulation with a 60 second time limit. If the problem cannot be solved to optimality, the best lower bound is reported. We repeat this process 10 times for each instance and method due to the inherent randomness of the algorithm.
In the second experiment, we run the CCG method using the sampled 5 most relevant scenarios as starting scenarios. For this experiment, we use a 600 second time limit and repeat each combination of instance and method 5 times (we use a smaller number of repetitions due to the increased computational effort). While the first experiment has the purpose to show the quality of the derived bound of the predicted scenarios, the second experiment also considers the subsequent performance of the algorithm. 

All experiments were conducted on a virtual server with Intel Xeon Gold 5220 CPU running at 2.20GHz. As a mixed-integer programming solver, we used Gurobi version 9.0.3. For better comparability, each run was restricted to one thread.

\subsection{Results for Traveling Salesperson Problems}

We first discuss the quality of prediction that DDH achieves. To this end, we use the 200 labeled instances of size 6 and 7, and use 90\% of each for training and 10\% of each for testing. In Figure~\ref{fig:tsproc}, we show the corresponding ROC curves. As expected, there is a small decline in performance when considering testing data instead of training data. Still, the AUC values of $0.80$ and $0.75$. respectively, indicate that the RFC method is indeed able to identify relevant scenarios.

\begin{figure}[htbp]
\begin{center}
\subfigure[Training data, $\text{AUC}=0.80$\label{fig:tsproc1}]{\includegraphics[width=0.37\textwidth]{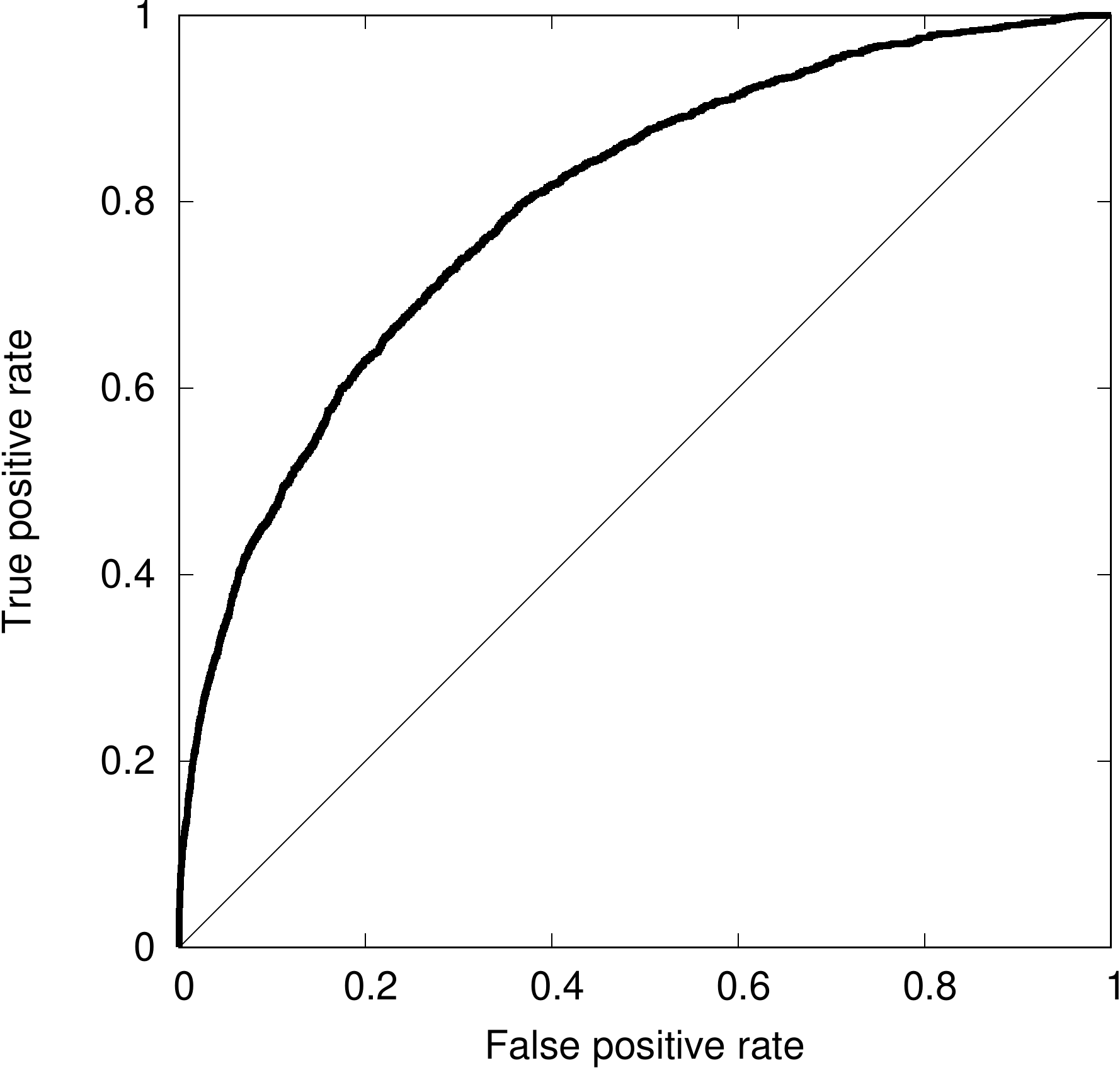}}%
\hspace{1cm}\subfigure[Test data, $\text{AUC}=0.75$\label{fig:tsproc2}]{\includegraphics[width=0.37\textwidth]{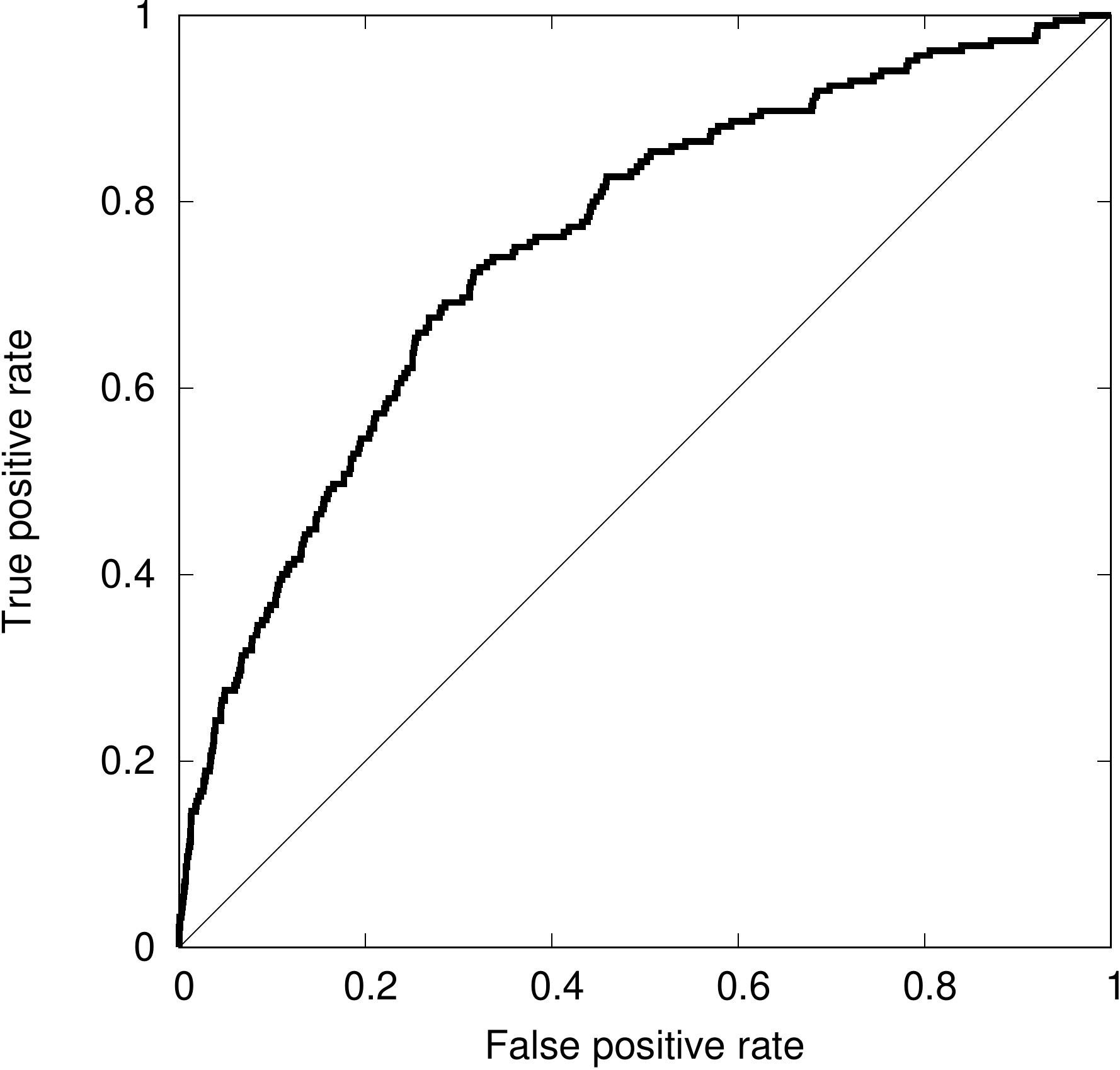}}
\end{center}
\caption{ROCs for TSP prediction.\label{fig:tsproc}}
\end{figure}

For the remaining experiments, we use all 200 labeled instances of size 6 and 7 for training and measure the performance of the predicted scenarios also for the larger-dimensional instances which were not used for training. We first compare the lower bounds using $k$ starting scenarios for $k\in\{1,3,\ldots,15\}$ in Figure~\ref{fig:tspgap}. We normalize the lower bounds so that the lower bound found by Random with one starting scenario is equal to one. We note that results are consistent for all problem sizes. In each case, Random results in the smallest lower bounds, followed by Maxsum, while DDH results in the largest bounds. In particular, we find that the DDH method is indeed scalable to problem sizes that are larger than the problem sizes used for training. However as expected the improvement decreases with the problem size.

\begin{figure}[htbp]
\begin{center}
\subfigure[6 nodes\label{fig:tsp6bound}]{\includegraphics[width=0.48\textwidth]{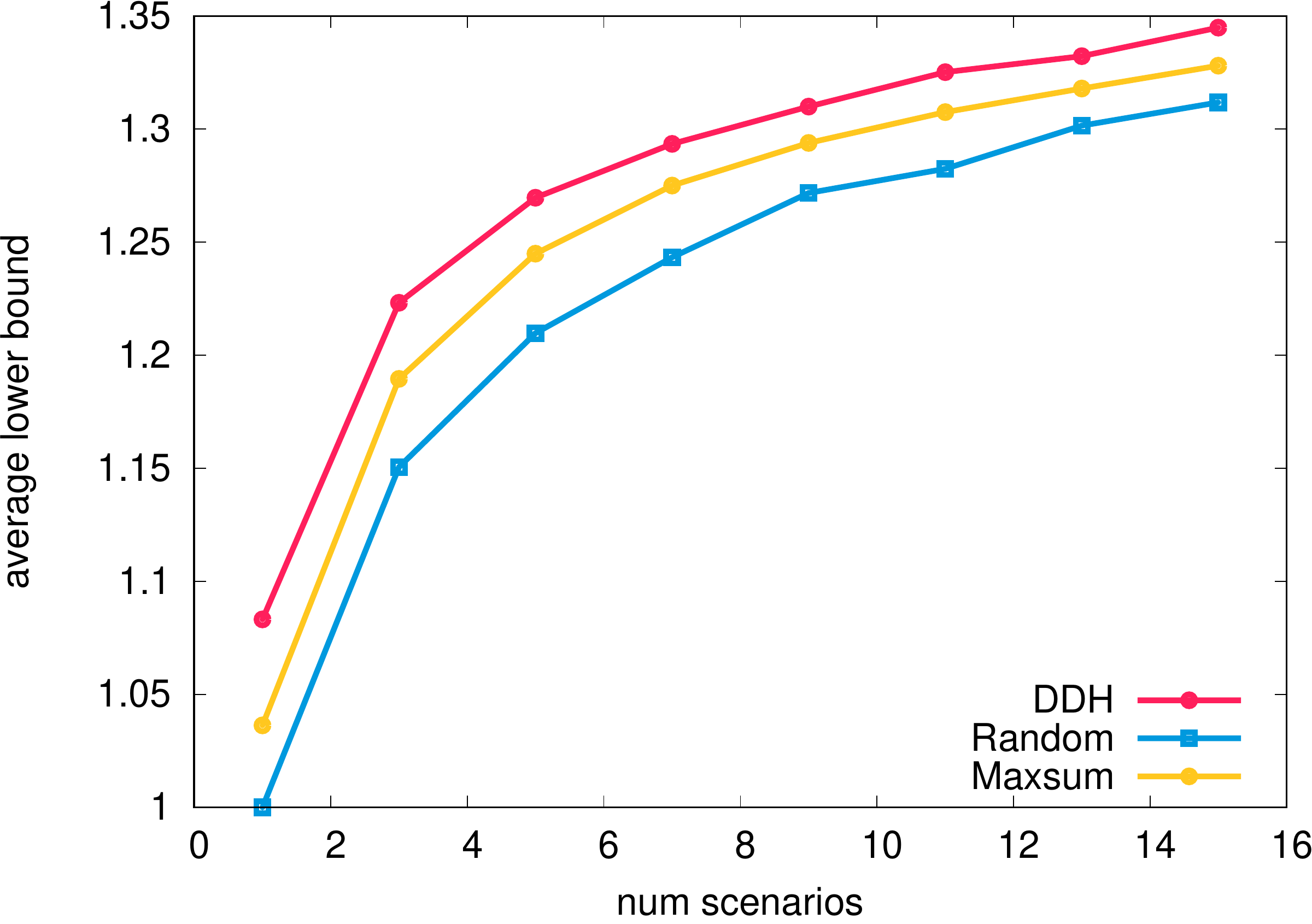}}%
\hfill
\subfigure[7 nodes\label{fig:tsp7bound}]{\includegraphics[width=0.48\textwidth]{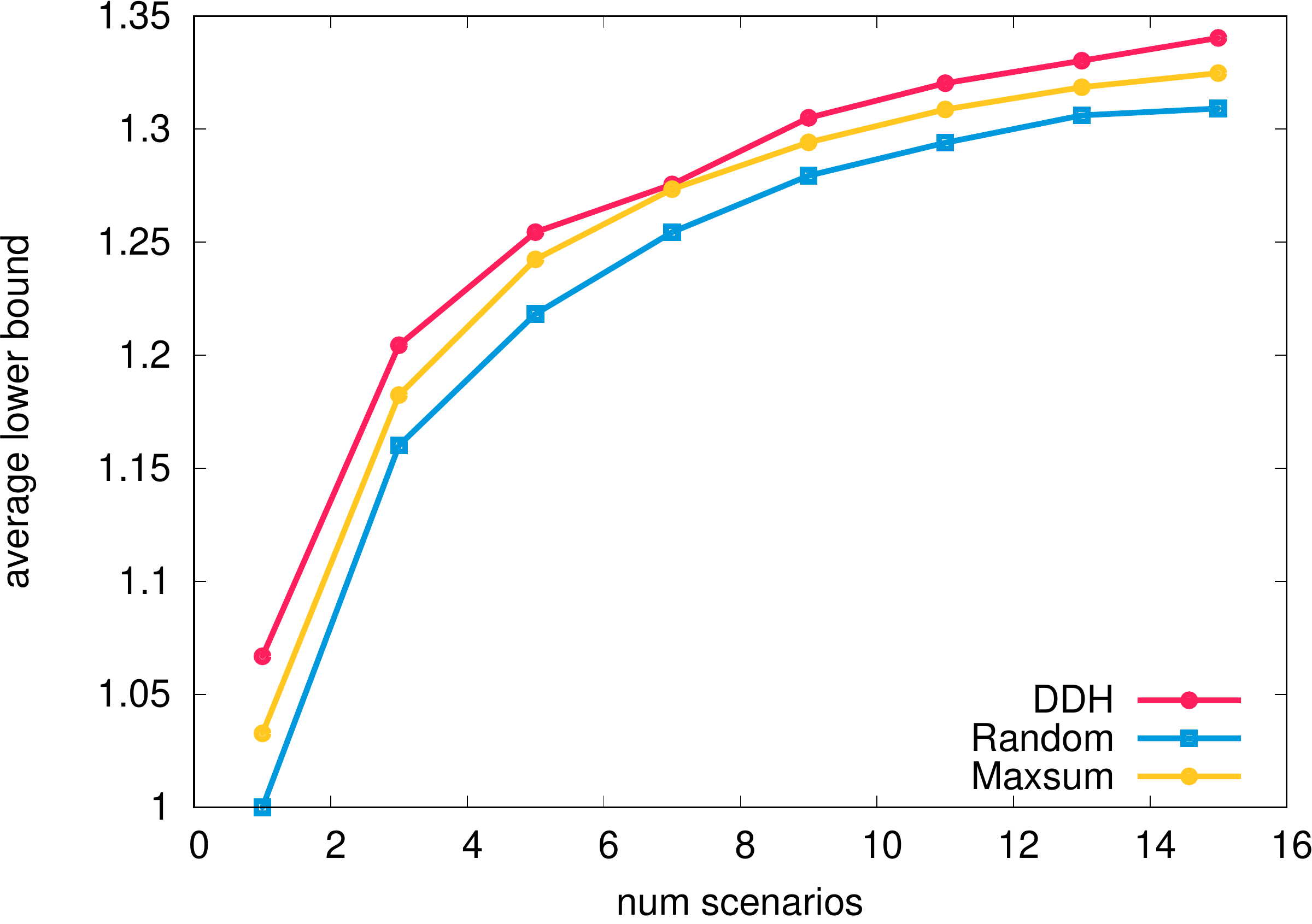}}
\subfigure[8 nodes\label{fig:tsp8bound}]{\includegraphics[width=0.48\textwidth]{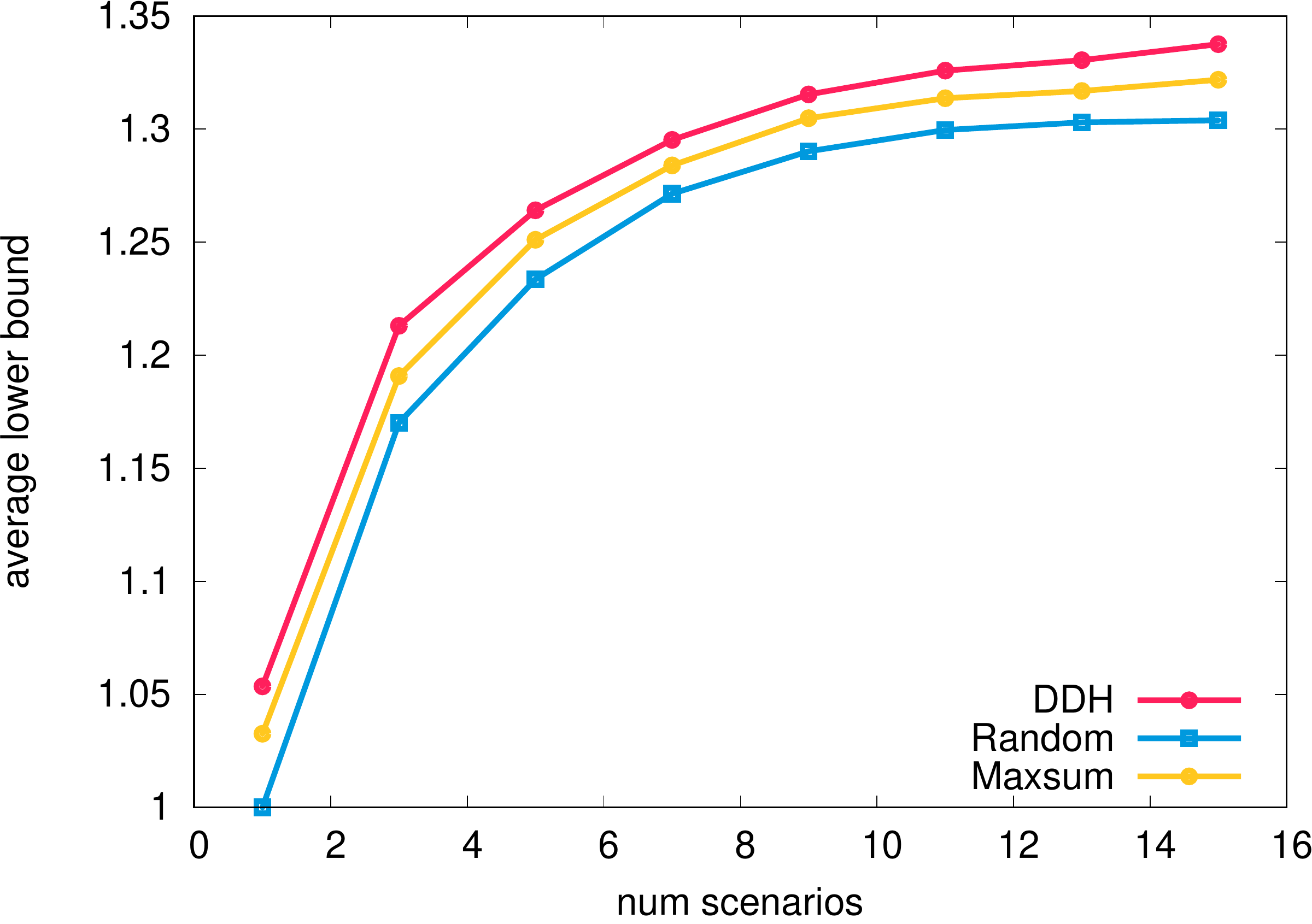}}%
\hfill
\subfigure[9 nodes\label{fig:tsp9bound}]{\includegraphics[width=0.48\textwidth]{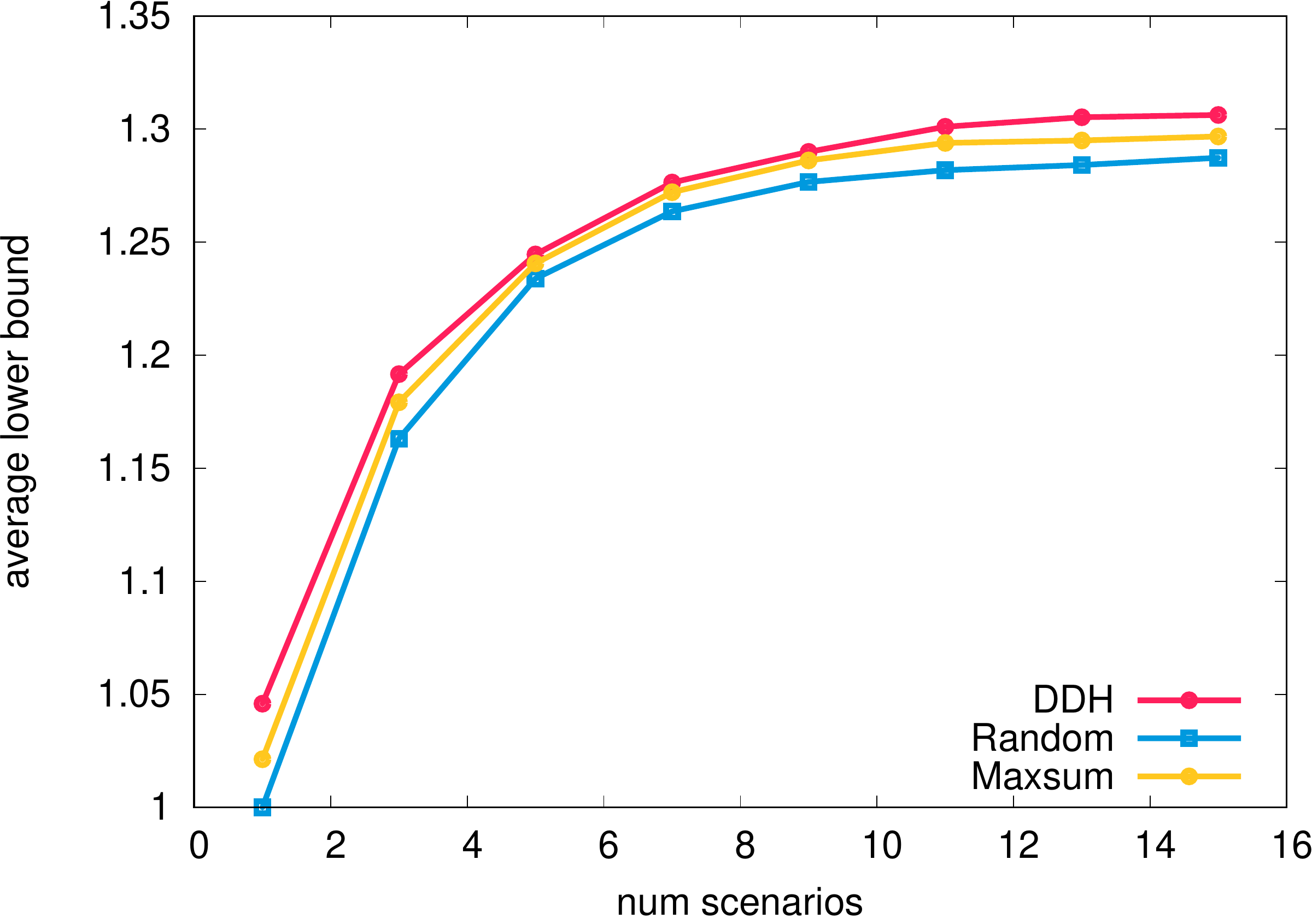}}
\end{center}
\caption{Average lower bounds depending on the number of starting scenarios for TSP. Values are normalized so that the lower bound of Random with one scenarios is equal to one.\label{fig:tspgap}}
\end{figure}

In the second experiment, we use 5 starting scenarios to run the CCG method. In Table~\ref{tab:tspexp2}, we report the lower bound found at the end of the time limit. We find results that are consistent with the first experiments; that is, a stronger lower bound at the start of the method also results in a better bound at the time limit. Indeed, the average lower bound for DDH is stronger than that of Random or Maxsum for each problem size.

\begin{table}[htbp]
\begin{center}
\begin{tabular}{r|rrr|rrr|rrr}
& \multicolumn{3}{c|}{Random} & \multicolumn{3}{c|}{Maxsum} & \multicolumn{3}{c}{DDH} \\
Nodes & Q1 & Avg & Q3 & Q1 & Avg & Q3 & Q1 & Avg & Q3 \\
\hline
6 & 335.3 & 342.0 & 348.9 & 333.0 & 342.1 & \textbf{350.8} & \textbf{336.0} & \textbf{342.3} & 350.1 \\
7 & 368.6 & 375.9 & 384.0 & \textbf{371.6} & 377.0 & 383.3 & 371.1 & \textbf{378.6} & \textbf{387.0} \\
8 & 405.2 & 411.8 & 416.2 & 406.9 & 412.2 & 417.1 & \textbf{408.6} & \textbf{415.6} & \textbf{421.0} \\
9 & 433.9 & 438.9 & 443.7 & 434.6 & 440.1 & 445.0 & \textbf{437.9} & \textbf{443.5} & \textbf{447.8}
\end{tabular}
\end{center}
\caption{Lower bounds after timelimit for TSP. Q1 is first quartile, Avg is average, Q3 is third quartile of data. Best results in each category is bold.}\label{tab:tspexp2}
\end{table}

In the Appendix we show the feature importance returned by the decision tree classifier. The results indicate that for TSP the features $f_{1,1}$ (average scenario entry), $f_{2,4}$ (scalar product of scenario and center scenario) and $f_{3,1}$ (deterministic optimal value for scenario) have the largest impact on the prediction performance. This shows that all three categories are useful for the prediction.

\subsection{Results for Shortest Path Problems}

In the same way as for TSP, we first show the prediction performance of the RFC on data with 5 and 6 layers, for which labels exist. In Figure~\ref{fig:sproc}, we show the ROC curves for training and test data, respectively. While the AUC values ($0.76$ and $0.68$) are slightly worse than before, it is still possible to predict relevant scenarios with some success.

\begin{figure}[htbp]
\begin{center}
\subfigure[Training data, $\text{AUC}=0.76$\label{fig:sproc1}]{\includegraphics[width=0.37\textwidth]{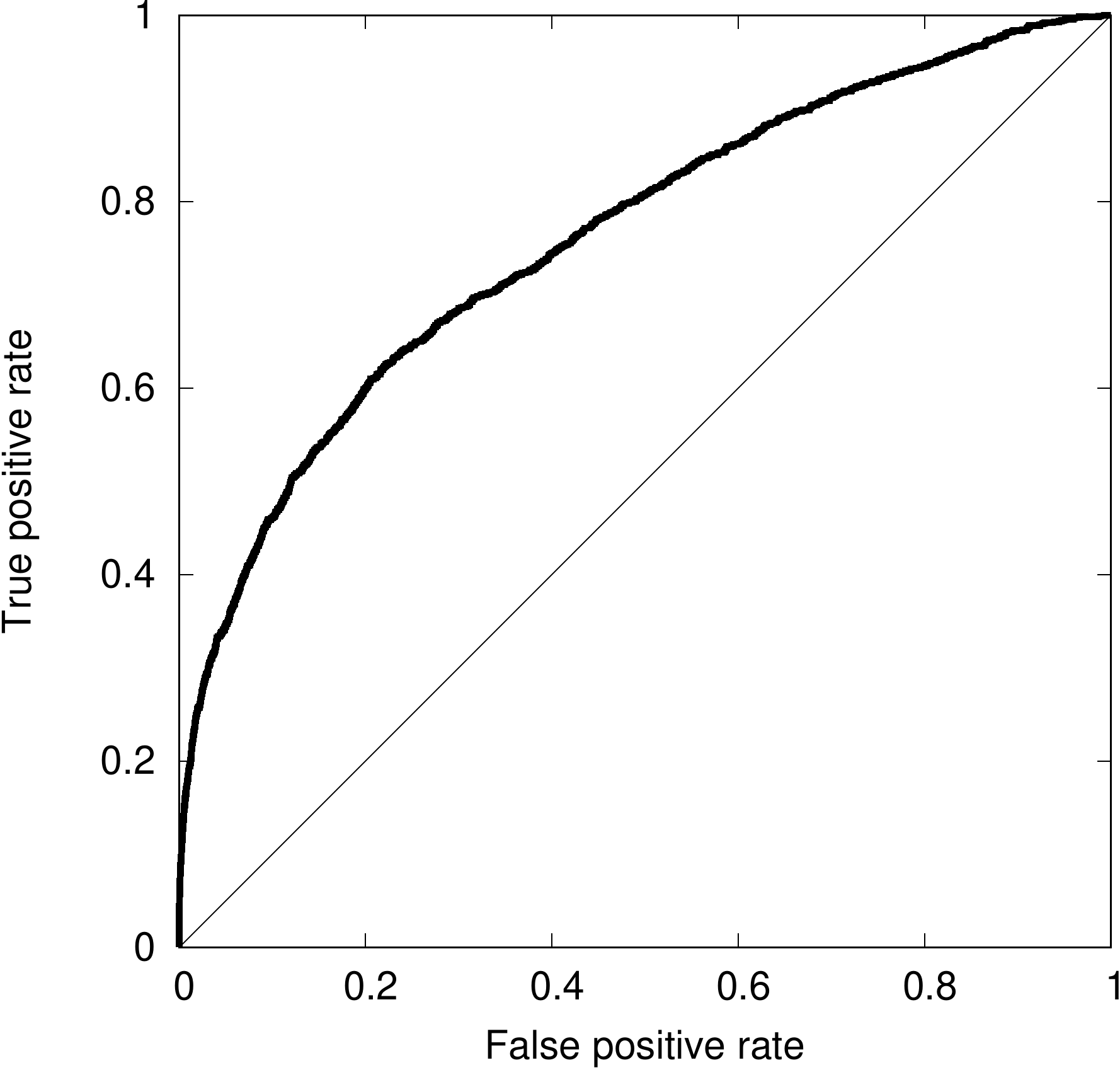}}%
\hspace{1cm}\subfigure[Test data, $\text{AUC}=0.68$\label{fig:sproc2}]{\includegraphics[width=0.37\textwidth]{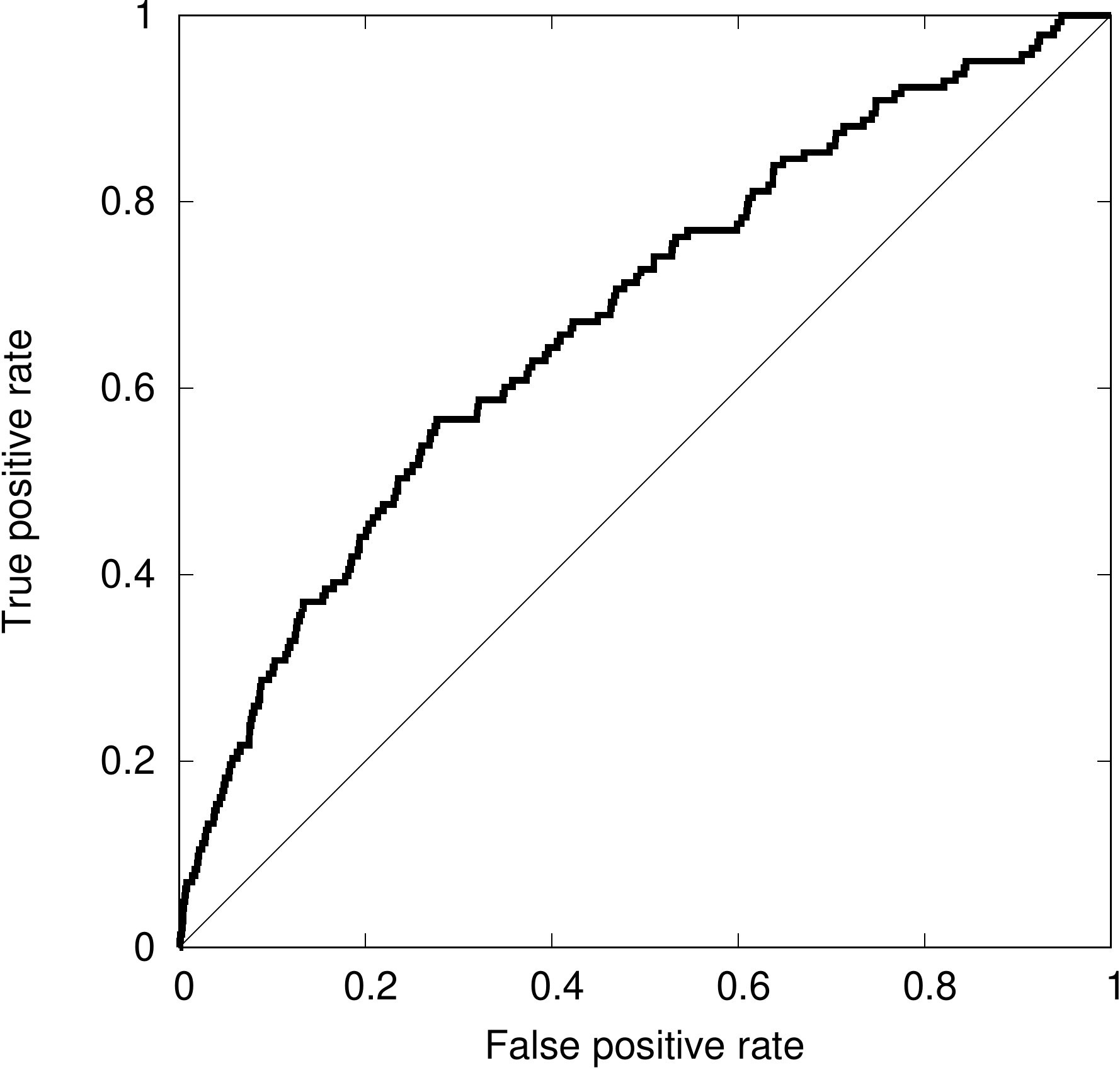}}
\end{center}
\caption{ROCs for SP prediction.\label{fig:sproc}}
\end{figure}

We show the results of the first experiment, i.e.,  the quality of the lower bounds for varying numbers of scenarios, in Figure~\ref{fig:spgap}. Notably, there is increased overlap between curves, in particular for Random and Maxsum. While for 5 and 6 layers, the performance of DDH is stronger than for the other methods, this advantage still exists for 7 and 8 layers, albeit less pronounced. Again, we find that it is possible to scale the prediction capabilities to larger problem sizes.
 
\begin{figure}[htbp]
\begin{center}
\subfigure[5 layers\label{fig:sp5bound}]{\includegraphics[width=0.48\textwidth]{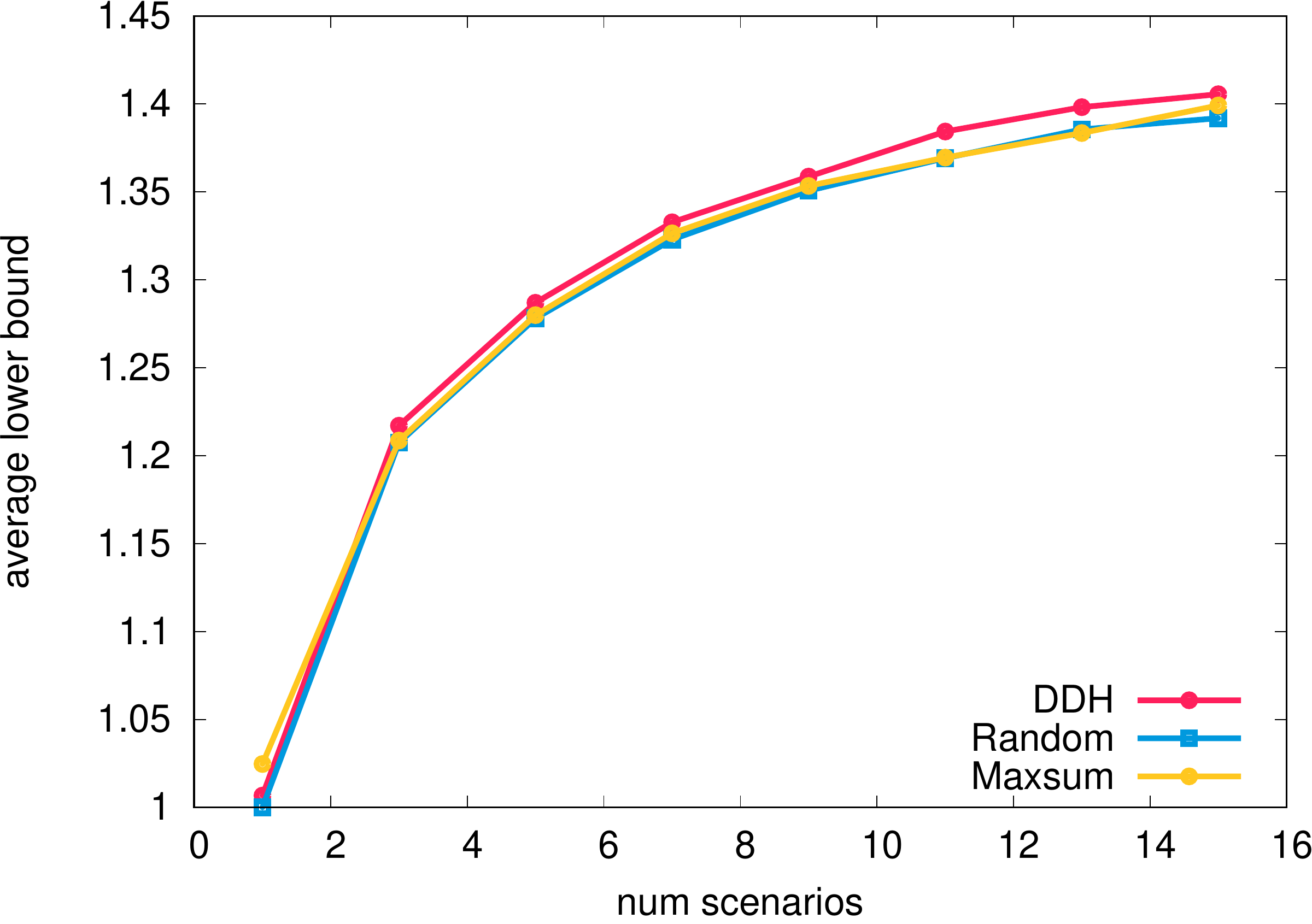}}%
\hfill
\subfigure[6 layers\label{fig:sp6bound}]{\includegraphics[width=0.48\textwidth]{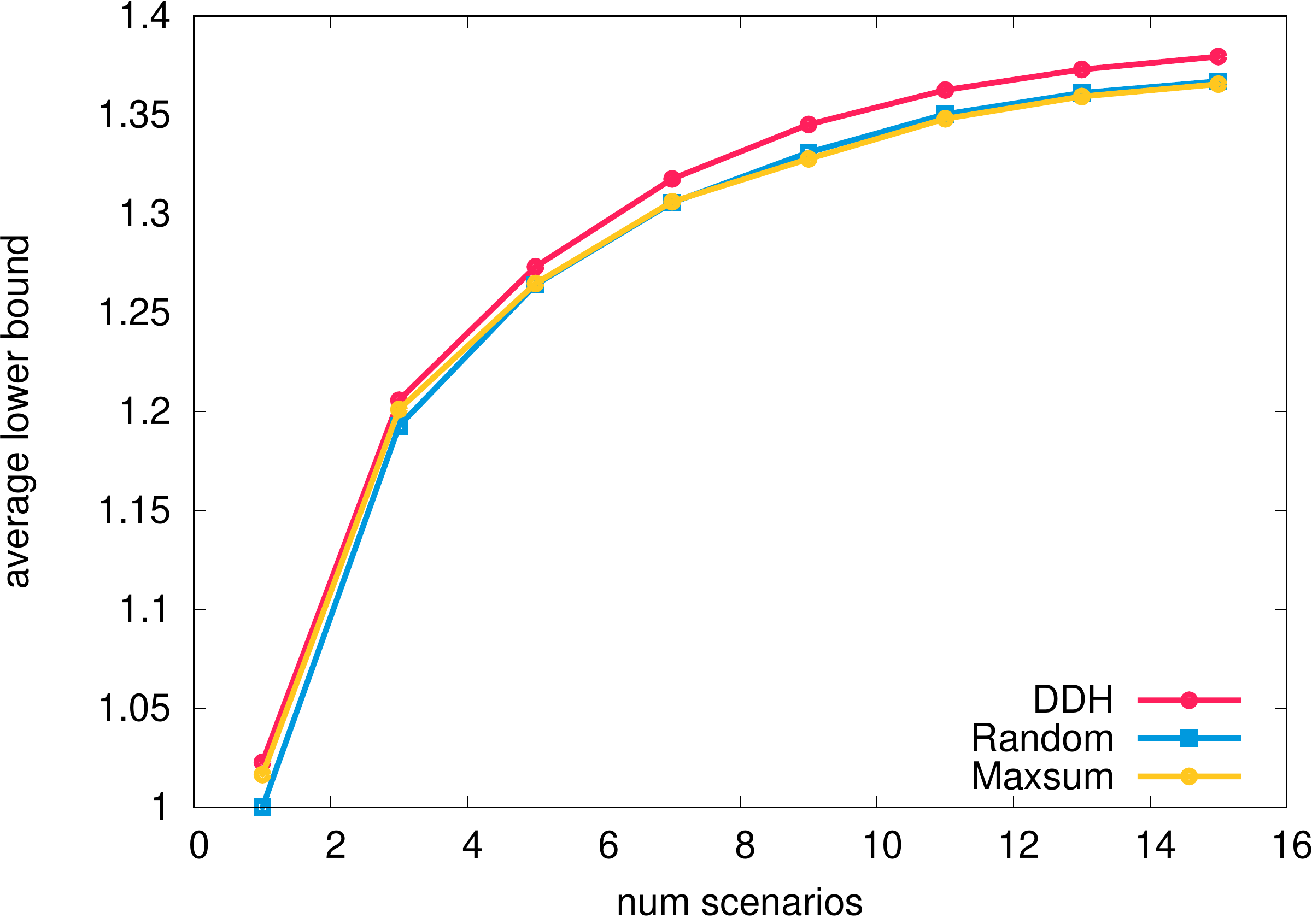}}
\subfigure[7 layers\label{fig:sp7bound}]{\includegraphics[width=0.48\textwidth]{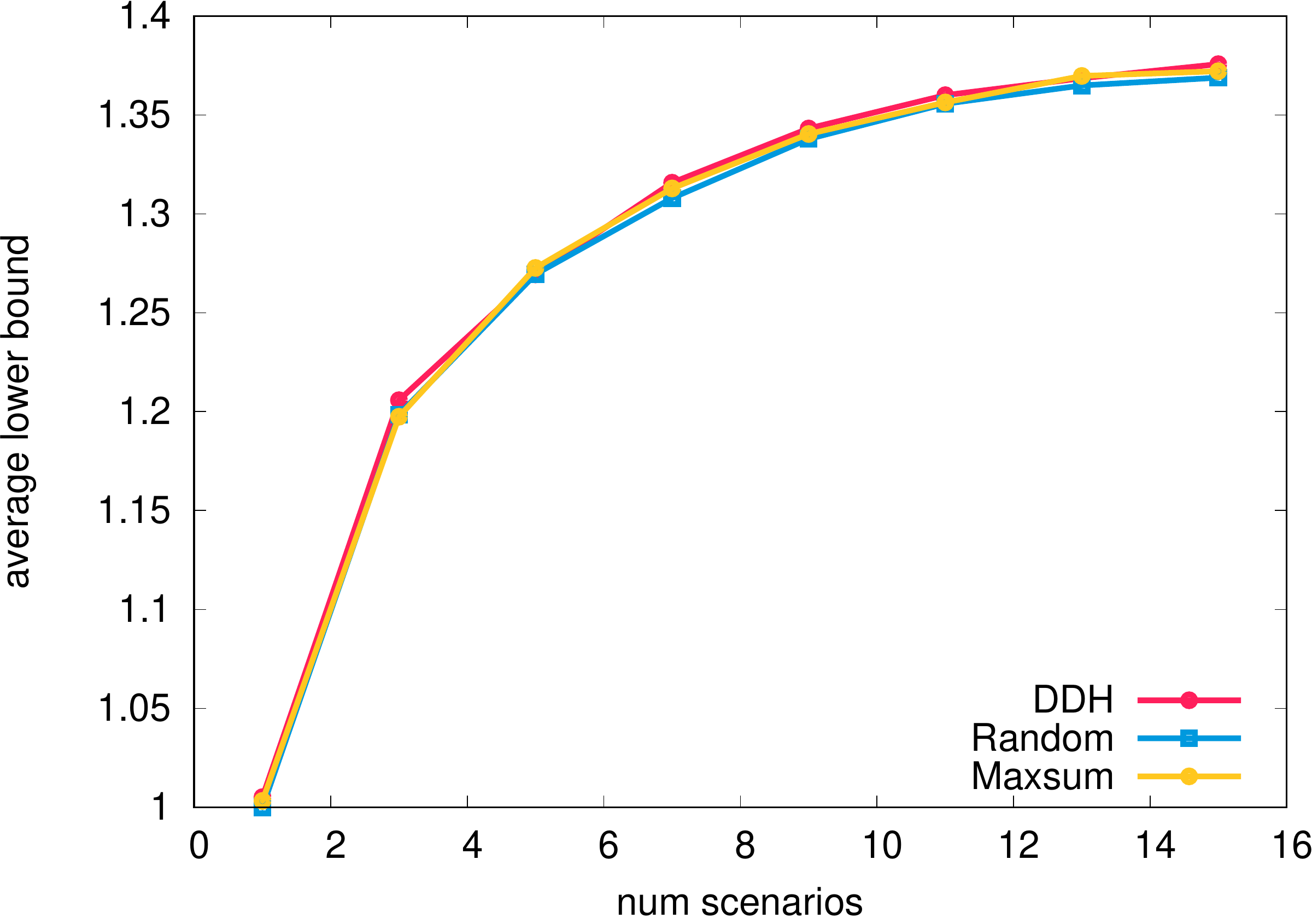}}%
\hfill
\subfigure[8 layers\label{fig:sp8bound}]{\includegraphics[width=0.48\textwidth]{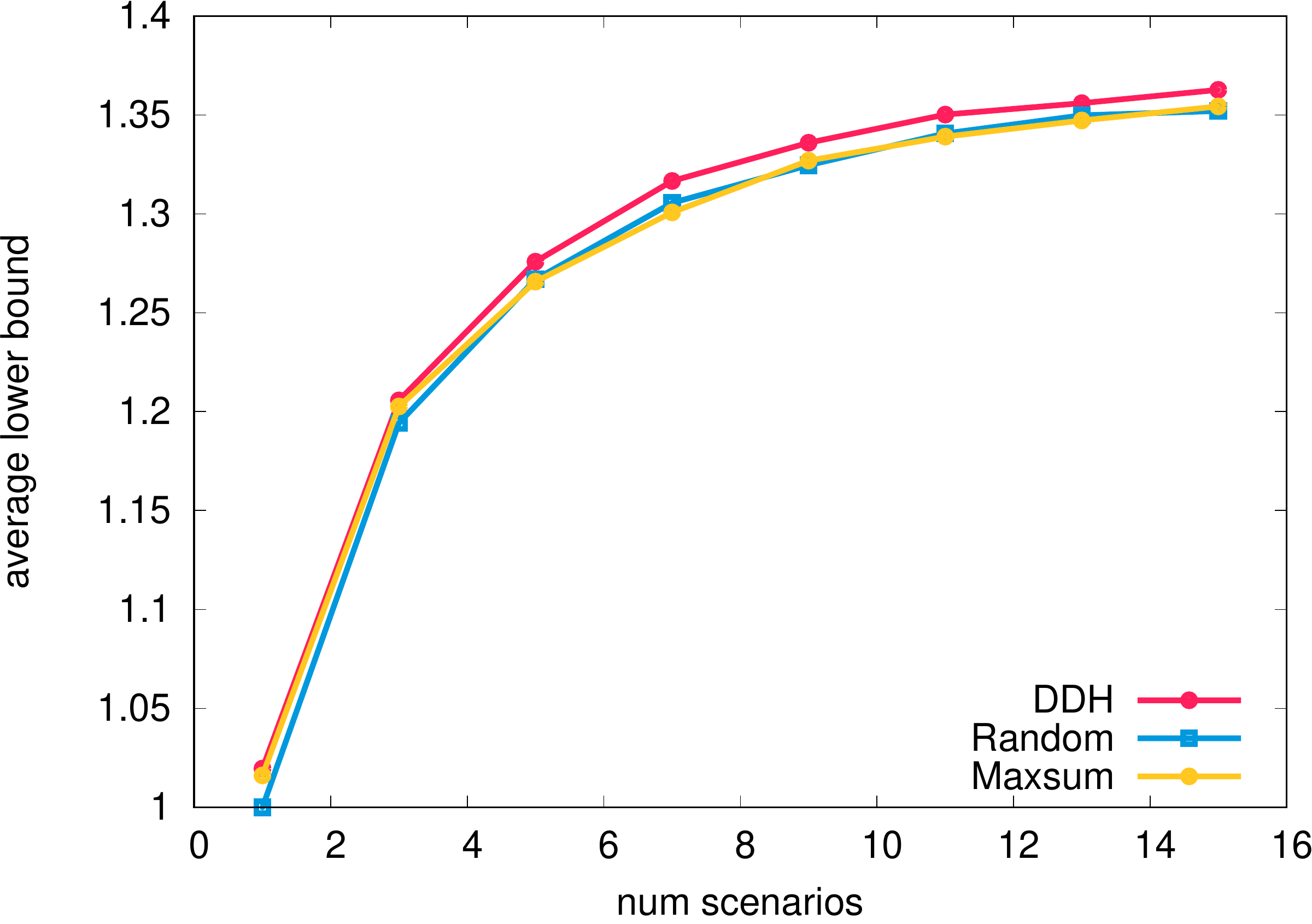}}
\end{center}
\caption{Average lower bounds depending on the number of starting scenarios for SP. Values are normalized so that the lower bound of Random with one scenarios is equal to one.\label{fig:spgap}}
\end{figure}

We show the lower bound after using CCG with 5 starting scenarios in Table~\ref{tab:spexp2}. While differences are less pronounced between methods in comparison to TSP, we can still note that DDH consistently outperforms both comparison methods.

\begin{table}[htbp]
\begin{center}
\begin{tabular}{r|rrr|rrr|rrr}
& \multicolumn{3}{c|}{Random} & \multicolumn{3}{c|}{Maxsum} & \multicolumn{3}{c}{DDH} \\
Nodes & Q1 & Avg & Q3 & Q1 & Avg & Q3 & Q1 & Avg & Q3 \\
\hline
5 & 352.1 & 351.8 & 350.0 & 358.7 & 358.4 & 358.9 & \textbf{364.9} & \textbf{363.4} & \textbf{363.8} \\
6 & 396.0 & 396.4 & 396.5 & 402.8 & 403.1 & 403.4 & \textbf{408.0} & \textbf{407.1} & \textbf{409.2} \\
7 & 443.5 & 443.9 & 444.7 & 449.8 & 450.5 & 451.1 & \textbf{455.7} & \textbf{457.1} & \textbf{459.1} \\
8 & 489.6 & 487.8 & 488.7 & 497.4 & 497.9 & 498.3 & \textbf{505.4} & \textbf{505.7} & \textbf{506.2}
\end{tabular}
\end{center}
\caption{Lower bounds after timelimit for SP. Q1 is first quartile, Avg is average, Q3 is third quartile of data. Best results in each category is bold.}\label{tab:spexp2}
\end{table}

To summarize the findings of these experiments, we note that DDH is an effective method in identifying relevant scenarios outperforming Random and Maxsum over all instance types and sizes. Furthermore, it is also efficient in the sense that we do not only get stronger lower bounds at the beginning of CCG, we also find better lower bounds at the end of our time limit. This even holds for instances that are larger than the instances used for training, which shows the practical applicability of our approach.

In the Appendix we show the feature importance returned by the decision tree classifier. The results indicate that for SP the features $f_{1,1}$ (average scenario entry), $f_{2,4}$ (scalar product of scenario and center scenario) and $f_{3,3},f_{3,4}$ (distance and quadratic distance to center in the problem-dependent distance measure) have the largest impact on the prediction performance. This shows that all three categories are useful for the prediction.

\section{Conclusions}

Generating constraints and variables is a staple method to solve robust optimization problems, due to its simplicity and wide applicability. Most commonly, the starting set of scenarios is generated without optimization effort, e.g. by choosing them randomly. To the best of our knowledge, this paper explores for the first time the potential of more sophisticated methods of producing starting scenarios. This leads to the Relevant Scenario Recognition Problem, which asks for a subset of scenarios that maximizes the bound provided by the master problem. We show that this problem is NP-hard, even if the robust problem can be solved in polynomial time. Furthermore solving this problem directly is only possible for small instance sizes; however, in this case, predicting starting scenarios is of less relevance.

We apply a data-driven machine learning method to predict relevant scenarios. We defined a set of instance features that do not depend on the dimension of the problem or the number of scenarios. This way, it is possible to generate labeled training data on small-dimensional problem instances, which may still be possible to be solved to optimality. Using a Random Forest Classifier, it is then possible to predict the relevance of scenarios for previously unseen problem instances, even of larger size. In computational experiments, we verified the efficacy of this approach. On two-stage traveling salesperson and shortest path problems, our data-driven prediction method outperformed alternative comparison methods consistently over all problem sizes, including on problems that are larger than those of the training set. The results indicate that designing dimension-independent features by using expert knowledge is valuable.

For future work it is desired to build a data collection of labeled instances which can be made publicly available.
Furthermore, it would be interesting to extend our work to convex uncertainty sets and constraint uncertainty.

\newpage
\appendix

\section*{Appendix}
\paragraph*{Proof of Theorem \ref{thm:exactIPformulationRO}}
 \ \\
In the following we assume that $X\subset \{ 0,1\}^n$ but the proof can be easily adjusted to the more general case that $X$ is an arbitrary finite set. We show that problem \eqref{eq:BSAP_RO} is equivalent to problem
\begin{align}
\max \ &\tau \label{eq:exactIPformulationRO}\\
s.t. \quad & \tau\le \sum_{i=1}^{m} z_{i}^x(c^i)^\top x  \quad \forall \ x\in X \label{constr:big-M}\\
& z_{i}^x\le u_i \quad \forall \ i\in [m], x\in X \label{constr:y_le_u}\\
& \sum_{i=1}^{m} z_i^x = 1 \quad \forall \ x\in X \label{constr:zSumto1}\\
& \sum_{i=1}^{m} u_i\le k \label{constr:kscenarios}\\
& u\in \{ 0,1\}^m, z^x\in \{0,1\}^{m} \ \forall \ x\in X .
\end{align}
The variables $u$ model the choice of scenarios, i.e. the optimal solution $u^*$ corresponds to the optimal index-set $\mathcal I^* = \left\{ i\in [m]: u_i=1\right\}$ of problem \eqref{eq:BSAP_RO}. Constraint \eqref{constr:kscenarios} ensures that at most $k$ scenarios can be chosen. The $z$-variables assign the worst-case scenario to each feasible solution $x\in X$. Constraints \eqref{constr:y_le_u} ensure that all $z$-variables are equal to zero if the scenario is not chosen. Otherwise it can be either zero or one. Constraints \eqref{constr:zSumto1} ensure that to each solution $x$ exactly one scenario $i$ has to be assigned. Constraints \eqref{constr:big-M} model the objective value of problem \eqref{eq:BSAP_RO} for a given choice $u$. Since we want to maximize $\tau$ in an optimal solution we always want to make the right-hand side as large as possible. Consider the constraint for a given $x\in X$ and the scenario $i\in [m]$ which was selected, i.e. for which we have $z_i^x=1$. In an optimal solution the $z$-variable will always choose the scenario $i$ which maximizes $(c^i)^\top x$ which proves the result.

\qed

The binary linear reformulation in the latter proof can be solved by any state-of-the-art integer programming solver as CPLEX or Gurobi. However the problem contains one constraint for each $x\in X$, hence the number of constraints can be exponential in the problem parameters if $X$ has exponential size. We can circumvent this problem by using an iterative constraint generation approach. Unfortunately, the Big-M constraints may still make the problem challenging to solve. To avoid these computational problems, we present an efficient data-driven heuristic to solve problem \eqref{eq:BSAP_RO} in Section \ref{sec:pearson}.

\paragraph*{Proof of Theorem \ref{thm:exactIPformulation2StageRO}}
 \ \\
The idea of the proof is similar to the one of Theorem \ref{thm:exactIPformulationRO} and hence the details are omitted. We show that problem \eqref{eq:BSAP_2StageRO} is equivalent to
\begin{align}
\max \ &\tau \label{eq:exactIPformulation2StageRO}\\
s.t. \quad & \tau\le c^\top x + \sum_{i\in [m]}z_i^y (d^i)^\top y^i \quad \forall \ x\in X, y=(y^1,\ldots ,y^m) \in R \label{constr:big-M2Stage}\\
& \sum_{i\in [m]} z_i^y = 1 \quad \forall y=(y^1,\ldots ,y^m)\in R \label{constr:t_sum_1}\\
& z_i^y\le u_i \quad \forall y=(y^1,\ldots ,y^m)\in R, \ i\in [m] \label{constr:t_le_u}\\
& \sum_{i=1}^{m} u_i\le k \label{constr:kscenarios2Stage}\\
& u\in \{ 0,1\}^m, z^y\in \{0,1\}^{m} \ \forall \ y=(y^1,\ldots ,y^m)\in R .
\end{align}
where $R=\left\{ (y^1,\ldots ,y^m): y^i\in Y, \  Ax + Dy^i\le b \ \forall \ i\in [m]\right\}$.
As in the proof of Theorem \ref{thm:exactIPformulationRO} the $u$-variables define the indices for the  solution $\mathcal I$ of problem \eqref{eq:BSAP_2StageRO}. Again constraints \eqref{constr:big-M2Stage} ensure that for each feasible solution the maximum scenario is selected. Constraints \eqref{constr:t_le_u} ensure that for the maximum we can only choose solutions $y^i$ and scenarios $d^i$ which are selected by the $u$-variables. \qed

\paragraph*{Feature Importance}

We list all 26 features we considered in Table~\ref{tab:features}, along with their importance as reported by the RFC model for TSP and SP. Recall that for TSP, prediction quality was slightly better than for SP. Indeed, we see that differences in feature importance is more nuanced. The three most important features for TSP are $f_{1,1}$ (average values), $f_{2,4}$ (scalar product with the center of the uncertainty set) and $f_{3,1}$ (optimal value of the deterministic problem). The three most important features for SP are $f_{3,1}$, $f_{3,4}$ (quadratic solution distance to center point) and $f_{1,1}$ as well.

\begin{table}[htbp]
\begin{center}
\begin{tabular}{l|rr}
Feature & TSP & SP \\
\hline
$f_{1,1}$ & 21.1 & 8.3 \\
$f_{1,2}$ & 2.3 & 2.7 \\
$f_{1,3}$ & 2.4 & 3.4 \\
\hline
$f_{2,1}$ & 3.9 & 5.6 \\
$f_{2,2}$ & 4.5 & 4.9 \\
$f_{2,3}$, $\kappa=1$ & 1.9 & 2.0 \\
$f_{2,3}$, $\kappa=2$ & 1.1 & 1.8 \\
$f_{2,3}$, $\kappa=3$ & 1.5 & 1.2 \\
$f_{2,3}$, $\kappa=4$ & 1.4 & 1.8 \\
$f_{2,3}$, $\kappa=5$ & 1.3 & 1.7 \\
$f_{2,4}$ & 15.7 & 6.6 \\
\hline
$f_{3,1}$ & 13.6 & 6.7 \\
$f_{3,2}$ & 0.7 & 1.2 \\
$f_{3,3}$ & 3.8 & 9.7 \\
$f_{3,4}$ & 3.9 & 9.7 \\
$f_{3,5}$, $\kappa=1$ & 0.7 & 1.1 \\
$f_{3,5}$, $\kappa=2$ & 1.0 & 1.1 \\
$f_{3,5}$, $\kappa=3$ & 1.1 & 1.7 \\
$f_{3,5}$, $\kappa=4$ & 1.4 & 1.3 \\
$f_{3,5}$, $\kappa=5$ & 1.0 & 1.6 \\
$f_{3,6}$ & 3.2 & 4.6 \\
$f_{3,7}$, $\alpha=1$ & 3.4 & 6.0 \\
$f_{3,7}$, $\alpha=2$ & 2.0 & 4.6 \\
$f_{3,7}$, $\alpha=3$ & 2.0 & 3.8 \\
$f_{3,7}$, $\alpha=4$ & 3.0 & 3.9 \\
$f_{3,7}$, $\alpha=5$ & 2.1 & 3.1
\end{tabular}
\end{center}
\caption{Feature importance in RFC prediction.}\label{tab:features}
\end{table}

\end{document}